\documentclass[11pt,english]{amsart}
\usepackage[T1]{fontenc}
\usepackage[latin9]{inputenc}
\usepackage{babel}
\usepackage{enumitem}
\usepackage{amsthm}
\usepackage{amssymb}
\usepackage[letterpaper]{geometry}
\usepackage[pdfusetitle,
 bookmarks=true,bookmarksnumbered=false,bookmarksopen=false,
 breaklinks=false,pdfborder={0 0 1},backref=false,colorlinks=false]
 {hyperref}
\hypersetup{
 colorlinks=true,linkcolor=teal,citecolor=purple,linktocpage=true}

\makeatletter
\theoremstyle{plain}
\newtheorem{thm}{\protect\theoremname}
\theoremstyle{definition}
\newtheorem{defn}[thm]{\protect\definitionname}
\theoremstyle{definition}
\newtheorem{convention}[thm]{Convention}
\theoremstyle{plain}
\newtheorem{lem}[thm]{\protect\lemmaname}
\theoremstyle{remark}
\newtheorem{rem}[thm]{\protect\remarkname}

\@ifundefined{date}{}{\date{}}
\usepackage{amsmath, tikz-cd, amssymb, placeins}
\usepackage{listings}
\DeclareRobustCommand*\cal{\@fontswitch\relax\mathcal}

\allowdisplaybreaks

\usetikzlibrary{calc}
\usetikzlibrary{decorations.pathmorphing}

\tikzset{curve/.style={settings={#1},to path={(\tikztostart)
    .. controls ($(\tikztostart)!\pv{pos}!(\tikztotarget)!\pv{height}!270:(\tikztotarget)$)
    and ($(\tikztostart)!1-\pv{pos}!(\tikztotarget)!\pv{height}!270:(\tikztotarget)$)
    .. (\tikztotarget)\tikztonodes}},
    settings/.code={\tikzset{quiver/.cd,#1}
        \def\pv##1{\pgfkeysvalueof{/tikz/quiver/##1}}},
    quiver/.cd,pos/.initial=0.35,height/.initial=0}

\tikzset{tail reversed/.code={\pgfsetarrowsstart{tikzcd to}}}
\tikzset{2tail/.code={\pgfsetarrowsstart{Implies[reversed]}}}
\tikzset{2tail reversed/.code={\pgfsetarrowsstart{Implies}}}
\tikzset{no body/.style={/tikz/dash pattern=on 0 off 1mm}}

\makeatother

\providecommand{\definitionname}{Definition}
\providecommand{\lemmaname}{Lemma}
\providecommand{\remarkname}{Remark}
\providecommand{\theoremname}{Theorem}

\begin{document}
\title{The Rasmussen $s$-invariant and exotic $4$-manifolds}
\author{Gheehyun Nahm}
\address{Department of Mathematics, Princeton University, Princeton, New Jersey
08544, USA}
\email{gn4470@math.princeton.edu}
\thanks{The author was partially supported by the ILJU Academy and Culture
Foundation, the Simons collaboration \emph{New structures in low-dimensional
topology}, and a Princeton Centennial Fellowship.}
\begin{abstract}
We give a short, analysis-free proof, inspired by Ren and Willis's,
of the existence of exotic compact, orientable $4$-manifolds. There
are two distinguishing features of our proof. First, we avoid skein
lasagna modules; we use Beliakova and Wehrli's generalization of Rasmussen's
$s$-invariant to links in $S^{3}$ directly. Second, we reduce the
complexity of the computations by choosing clever induction hypotheses
in Sto\v{s}i\'{c}'s induction scheme; this in particular allows
us to avoid Ren and Willis's Comparison Lemma.
\end{abstract}

\maketitle
In this note we give a short, analysis-free proof, inspired by Ren
and Willis's recent breakthrough \cite{ren2024khovanovv3}, of the
following theorem.
\begin{thm}
\label{thm:mainthm}The homeomorphic knot traces $X_{1}(5_{2})$ and
$X_{1}(P(-3,3,8))$ are not diffeomorphic.
\end{thm}

Theorem~\ref{thm:mainthm} was first proved by Akbulut \cite{MR1094460,MR1094459},
who distinguished the exotic pair via gauge theory. Remarkably, Ren
and Willis distinguished them combinatorially; they used the $\mathfrak{gl}_{2}$
Khovanov--Rozansky skein lasagna module defined by Morrison, Walker,
and Wedrich \cite{MWW}. In this note, we give a short, skein lasagna-free
proof of Theorem~\ref{thm:mainthm} as follows. First, we give a
direct proof of Theorem~\ref{thm:mainthm} assuming Theorem~\ref{thm:sinv},
which is the main technical input from Khovanov homology. Then, we
give a simpler proof of Theorem~\ref{thm:sinv}.
\begin{defn}
Given a knot $K$ in $S^{3}$, $f\in\mathbb{Z}$, and $n\in\mathbb{Z}_{\ge0}$,
let $K_{n}^{f}$ denote $f$-framed $(2n+1)$\nobreakdash-cables
of $K$, where $(n+1)$ strands are oriented in one way and the remaining
$n$ strands are oriented in the other way.
\end{defn}

\begin{thm}[{\cite[Theorems~1.13~and~5.5]{ren2024khovanovv3}}]
\label{thm:sinv}If $K$ is a negative knot, then $s(K_{n}^{1})=s(K)-2n$.
\end{thm}

\begin{proof}[Proof of Theorem~\ref{thm:mainthm} assuming Theorem~\ref{thm:sinv}]
On the one hand, since $P(-3,3,8)$ is slice, by capping off a slice
disk with the core of the $2$-handle, we see that there exists a
smooth $2$-sphere that represents a generator of $H_{2}(X_{1}(P(-3,3,8)))\cong\mathbb{Z}$.

On the other hand, we claim that there does not exist a smooth $2$-sphere
$S$ that represents a generator of $H_{2}(X_{1}(5_{2}))\cong\mathbb{Z}$.
Assume that there exists such a $2$-sphere $S\subset X_{1}(5_{2})$.
Isotope $S$ so that $S$ intersects the cocore $CC$ of $X_{1}(5_{2})$
transversely; then $S$ and $CC$ intersect positively at $(n+1)$
points and negatively at $n$ points for some $n\in\mathbb{Z}_{\ge0}$.
Let $p\in S$ be a point away from $CC$, and let $N(CC\sqcup\{p\})$
be an open tubular neighborhood of $CC\sqcup\{p\}$. Then, $X_{1}(5_{2})\setminus N(CC\sqcup\{p\})$
is diffeomorphic to $[0,1]\times S^{3}$, and the surface $S$ restricted
to this complement is a connected, oriented link cobordism in $[0,1]\times S^{3}$
from the unknot $U$ to the link $K_{n}^{1}$ for $K=5_{2}$. This
link cobordism is topologically $S^{2}\setminus(\sqcup^{2n+2}\mathring{D^{2}})$,
and so it has Euler characteristic $-2n$. However, since $5_{2}$
is a negative knot, $s(K_{n}^{1})=s(5_{2})-2n=-2-2n$ by Theorem~\ref{thm:sinv},
and since $-2-2n<-2n$, this contradicts the slice genus bound from
\cite[Proposition~7.6~(6)]{MR4541332} (compare~\cite{Ra,MR2462446}).
\end{proof}
The rest of this note is organized as follows. First, we establish
our conventions. Then, we reduce Theorem~\ref{thm:sinv} to Lemma~\ref{lem:band-map}.
Finally, we prove Lemma~\ref{lem:band-map} which takes up most of
this note.
\begin{convention}
We work over a field of characteristic $\neq2$. Alternatively, we
may work over any field and consider the Bar-Natan deformation \cite{BN1}
instead of the Lee deformation \cite{MR2173845}. Khovanov homology
\cite{MR1740682} is bigraded. If $A$ is bigraded, denote its homological
grading $h$ and quantum grading $q$ summand as $A^{h,q}$. Denote
its homological grading $h$ summand as $A^{h}$.

If $L$ is an oriented link, we denote $L$ together with an unlinked
unknot as $L\sqcup U$.

Let $L$ be an oriented link, and let $\mathfrak{o}$ be an orientation
on $L$. Denote as $x_{\mathfrak{o}}\in Kh_{Lee}(L)$ the corresponding
Lee generator \cite{MR2173845}, defined up to multiplication by a
nonzero element of the base field. If $\mathfrak{o}$ is the orientation
of $L$, write $x_{L}:=x_{\mathfrak{o}}$; the $s$-invariant is $s(L):={\rm gr}_{q}(x_{L})+1$.

Recall \cite{Ra} that if $B:L\to L'$ is an oriented saddle cobordism
(i.e.\ a split or merge band on $L$) between oriented links $L,L'$,
then the induced map (\emph{``band map''}) $Kh_{Lee}(L)\to Kh_{Lee}(L')$
(defined up to sign) maps $x_{L}$ to a nonzero multiple of $x_{L'}$.
Also, if $B:L\to L'$ is a nonorientable saddle cobordism, then the
induced band map $Kh_{Lee}(L)\to Kh_{Lee}(L')$ is identically zero;
compare \cite{MR4504654}.
\end{convention}

\label{pagelabel}Let us reduce Theorem~\ref{thm:sinv} to the following
lemma (Lemma~\ref{lem:band-map}). Consider a local merge band between
two oppositely oriented strands of $K_{m+1}^{1}$; doing a band surgery
yields $K_{m}^{1}$ together with an unlinked unknot $U$, which we
denote as $K_{m}^{1}\sqcup U$. Let us consider the dual band on $K_{m}^{1}\sqcup U$;
then doing a band surgery yields $K_{m+1}^{1}$.
\begin{lem}[Main lemma]
\label{lem:band-map}For all $m\ge0$, ${\rm dim}Kh^{0}(K_{m}^{1})={\rm dim}Kh_{Lee}^{0}(K_{m}^{1})$,
and the band map $\varphi:Kh^{0}(K_{m}^{1}\sqcup U)\to Kh^{0}(K_{m+1}^{1})$
is injective.
\end{lem}

\begin{proof}[Proof of Theorem~\ref{thm:sinv} assuming Lemma~\ref{lem:band-map}]
Let us show that ${\rm gr}_{q}(x_{K_{m+1}^{1}})={\rm gr}_{q}(x_{K_{m}^{1}})-2$;
then Theorem~\ref{thm:sinv} follows. Let $\varphi_{Lee}:Kh_{Lee}(K_{m}^{1}\sqcup U)\to Kh_{Lee}(K_{m+1}^{1})$
be the band map on Lee homology, and let $\varphi_{\infty}$ be the
induced map between the associated graded objects ${\rm gr}Kh_{Lee}$
of Lee homology (i.e.\ the $E_{\infty}$ page of the Lee spectral
sequence). Since ${\rm gr}_{q}(x_{K_{m}^{1}\sqcup U})={\rm gr}_{q}(x_{K_{m}^{1}})-1$
and $\mathrm{gr}_{q}\varphi_{\infty}=\mathrm{gr}_{q}\varphi=-1$,
it is sufficient to show that $\varphi_{\infty}(x_{K_{m}^{1}\sqcup U})\neq0$,
where we abuse notation and denote the image of $x_{K_{m}^{1}\sqcup U}$
in ${\rm gr}Kh_{Lee}(K_{m}^{1}\sqcup U)$ also as $x_{K_{m}^{1}\sqcup U}$.
By the first part of Lemma~\ref{lem:band-map}, for both $K_{m}^{1}\sqcup U$
and $K_{m+1}^{1}$, the Lee spectral sequence collapses on the $E_{1}$
page in homological grading~$0$. Hence, by the second part of Lemma~\ref{lem:band-map},
$\varphi_{\infty}(x_{K_{m}^{1}\sqcup U})\neq0$.
\end{proof}
As in \cite[Theorem~5.5~and~Appendix~A]{ren2024khovanovv3}, we induct
on a family of links to show Lemma~\ref{lem:band-map}, based on
Sto\v{s}i\'{c}'s induction scheme \cite{MR2308944,MR2492301}; compare
\cite{MR3116306,MR4843752}. Here, we choose clever induction hypotheses
to avoid using the Comparison Lemma~\cite[Lemma~5.8]{ren2024khovanovv3}
and most of the counting arguments; the most technical combinatorial
argument we need is \cite[Figures~15--20]{MR3116306} (i.e.\ \cite[Lemma~A.3~(1)~and~(2)]{ren2024khovanovv3}).

\begin{defn}[Auxiliary links for the induction scheme]
For $m,a,i\in\mathbb{Z}_{\ge0}$ such that $i\le2m$, let $D_{m,a,i}$
be the braid closure of the (positive) braid on $(2m+1)$ strands
$(\sigma_{1}\cdots\sigma_{2m})^{a}\sigma_{1}\cdots\sigma_{i}$.\footnote{This is denoted as $D_{2m+1,a}^{i}$ in \cite[Appendix A]{ren2024khovanovv3}.
Our convention for braid words agrees with \cite{MR4843752,ren2024khovanovv3};
see \cite[Figure~3]{MR4843752}.} Let $D_{m,a,i}'$ be its \emph{flip}, i.e.\ the braid closure of
$\sigma_{2m+1-i}\cdots\sigma_{2m}(\sigma_{1}\cdots\sigma_{2m})^{a}$.
For $f\in\mathbb{Z}$, denote as $K_{m,a,i}^{f}$ (resp.\ $K_{m,a,i}^{f}{}')$
oriented diagrams given by inserting $D_{m,a,i}$ (resp.\ $D_{m,a,i}'$)
into an $f$-framed diagram of $K$. Here the strands of the braid
are oriented the same as $K$.
\end{defn}

\begin{defn}[Renormalization for the induction scheme]
For oriented diagrams $Z$ (such as $K_{m,a,i}^{f}$ and $K_{m,a,i}^{f}{}'$)
given by inserting the braid closure of a braid on $(2m+1)$ strands
into some diagram of $K$, shift the homological gradings of $Kh(Z)$
and $Kh_{Lee}(Z)$ by the same amount such that the Lee generator
$x_{Z}$ is in homological grading $-(2m+1)^{2}/2+1/2$. Write these
renormalized groups as $\overline{Kh}(Z)$ and $\overline{Kh}_{Lee}(Z)$.

Renormalize $Kh(Z\sqcup U)$ (resp.\ $Kh_{Lee}(Z\sqcup U)$) in homological
grading such that it agrees with $\overline{Kh}(Z)\otimes Kh(U)$,
(resp.\ $\overline{Kh}_{Lee}(Z)\otimes Kh_{Lee}(U)$). Denote it
as $\overline{Kh}(Z\sqcup U)$ (resp.\ $\overline{Kh}_{Lee}(Z\sqcup U)$).
\end{defn}

Before we start the induction argument, let us do some bookkeeping
(Lemma~\ref{lem:renormalize}).
\begin{lem}[{Homological gradings of Lee generators, \cite[Proposition 2.8 (3)]{ren2024khovanovv3}}]
\label{lem:hgrading}Let $L$ be a link, and let $\mathfrak{o},\mathfrak{o}'$
be two orientations on $L$. Let $L_{+}$ be the sublink of $L$ consisting
of the components where $\mathfrak{o}$ agrees with $\mathfrak{o}'$.
Let $D_{L}$ be a diagram for $L$, and denote the writhe of the oriented
diagram $(D_{L},\mathfrak{o})$ as $w(D_{L},\mathfrak{o})$. Then,
\[
\mathrm{gr}_{h}(x_{\mathfrak{o}})-\mathrm{gr}_{h}(x_{\mathfrak{o}'})=2\ell k((L_{+},\mathfrak{o}'),(L\setminus L_{+},\mathfrak{o}'))=(w(D_{L},\mathfrak{o}')-w(D_{L},\mathfrak{o}))/2.
\]
\end{lem}

\begin{proof}
Immediate from the definitions; beware the difference with \cite{ren2024khovanovv3}
in conventions.
\end{proof}
\begin{lem}[Orientation $\mathfrak{o}_{p}$]
\label{lem:orientation-op}Let $f\in\mathbb{Z}$ and $m\in\mathbb{Z}_{\ge0}$,
and let $L$ be an $f$-framed $(2m+1)$-cable of a knot $K$. For
$p\in\mathbb{Z}$, $0\le p\le2m+1$, define $\mathfrak{o}_{p}$ as
the orientation of $L$ where $p$ strands are oriented the same as
$K$ and the other $(2m+1-p)$ strands are oriented in the other way.
Then, 
\[
{\rm gr}_{h}(x_{\mathfrak{o}_{p}})-{\rm gr}_{h}(x_{\mathfrak{o}_{q}})=f((2m+1-2q)^{2}-(2m+1-2p)^{2})/2.
\]
\end{lem}

\begin{proof}
Consider a diagram of $L$ where all the strands have framing $f$.
Then $w(K_{m}^{f},\mathfrak{o}_{p})=f(2m+1-2p)^{2}$, and so the lemma
follows by Lemma~\ref{lem:hgrading}.
\end{proof}
\begin{lem}
\label{lem:renormalize}For all $m\ge0$, $Kh(K_{m}^{1})\cong\overline{Kh}(K_{m,2m,2m}^{0})$
as homologically $\mathbb{Z}$-graded groups.
\end{lem}

\begin{proof}
As unoriented links, both $K_{m,2m,2m}^{0}$ and $K_{m}^{1}$ are
$1$-framed $(2m+1)$-cables of $K$, and so we are left to check
that $\overline{Kh}(K_{m,2m,2m}^{0})$ is renormalized such that its
homological grading agrees with that of $Kh(K_{m}^{1})$. Note that
$x_{K_{m}^{1}}=x_{\mathfrak{o}_{m}}$ and $x_{K_{m,2m,2m}^{0}}=x_{\mathfrak{o}_{2m+1}}$.
By Lemma~\ref{lem:orientation-op}, ${\rm gr}_{h}(x_{\mathfrak{o}_{m}})-{\rm gr}_{h}(x_{\mathfrak{o}_{2m+1}})=((2m+1)^{2}-1)/2$.
We renormalized $\overline{Kh}_{Lee}(K_{m,2m,2m}^{0})$ such that
${\rm gr}_{h}(x_{\mathfrak{o}_{2m+1}})=-(2m+1)^{2}/2+1/2$, and so
${\rm gr}_{h}(x_{\mathfrak{o}_{m}})=0$ in $\overline{Kh}_{Lee}(K_{m,2m,2m}^{0})$.
Since ${\rm gr}_{h}(x_{L})=0$ in $Kh_{Lee}(L)$ for any $L$ and
$x_{K_{m}^{1}}=x_{\mathfrak{o}_{m}}$, ${\rm gr}_{h}(x_{\mathfrak{o}_{m}})=0$
in $Kh_{Lee}(K_{m}^{1})$ as well.
\end{proof}
Let us start the induction argument. Fix $w\in\mathbb{Z}_{\le0}$.
We induct on quadruples of integers $(f,m,a,i)\in\mathrm{Ind}$, where
\[
{\rm Ind}:=\{(f,m,a,i)\in\mathbb{Z}^{4}:w\le f\le0,\ m\ge0,\ 0\le a,i\le2m\}
\]
is given the lexicographic order, and show the following two statements
for all knots\footnote{For each $K$, we in fact only need to consider $K$ and its orientation
reverse $rK$. Indeed, what we really want to show are Statements~(\ref{enu:induct1})~and~(\ref{enu:induct2})
for $K_{m,a,i}^{f}$ and $K_{m,a,i}^{f}{}'$; for simplicity of exposition,
we only consider $K_{m,a,i}^{f}$ and use the identity $K_{m,a,i}^{f}{}'=r((rK)_{m,a,i}^{f})$.} $K$ that admit a negative diagram of writhe $w\in\mathbb{Z}_{\le0}$:\footnote{Of course, instead of considering $K_{m,a,i}^{f}$ and inducting on
${\rm Ind}$, we may consider $K_{m,(f-w)(2m+1)+a,i}^{w}$ and induct
on $\{(m,a,i)\in\mathbb{Z}^{3}:m,a\ge0,\ 0\le i\le2m\}$, but we find
the former notationally simpler.}
\begin{enumerate}[label=(\Alph{enumi}), ref=\Alph{enumi}]
\item \label{enu:induct1}$\overline{Kh}^{h}(K_{m,a,i}^{f})=0$ for all
$h>0$
\item \label{enu:induct2}$\dim\overline{Kh}^{0}(K_{m,a,i}^{f})=\dim\overline{Kh}_{Lee}^{0}(K_{m,a,i}^{f})$
\end{enumerate}
By Lemma~\ref{lem:renormalize}, the first statement of Lemma~\ref{lem:band-map}
is precisely Statement~(\ref{enu:induct2}) for $(f,m,a,i)=(0,m,2m,2m)$.
We will show the second statement of Lemma~\ref{lem:band-map} at
the end.

The base cases of the induction are $(f,m,a,i)=(w,m,0,0)$. Since
$K$ has a negative diagram of writhe $w$, $K_{m,0,0}^{w}$ is a
negative link. Hence, the maximum homological grading of $\overline{Kh}(K_{m,0,0}^{w})$
corresponds to the oriented resolution, and so it is the homological
grading of the Lee generator, which by definition is $-(2m+1)^{2}/2+1/2$.
Statement~(\ref{enu:induct1}) holds since $-(2m+1)^{2}/2+1/2\le0$.
Statement~(\ref{enu:induct2}) trivially holds if $m>0$, and if
$m=0$, then Statement~(\ref{enu:induct2}) holds since $K_{0,0,0}^{w}=K$
is a negative knot and so the dimensions are both $2$ (compare \cite[Proposition~6.1]{MR2034399}).

Let us carry out the induction step. First, the cases where $i=0$
follow since as oriented links, $K_{m,a,0}^{f}$ and $K_{m,a-1,2m}^{f}$
are isotopic for $a\ge1$, and $K_{m,0,0}^{f+1}$ and $K_{m,2m,2m}^{f}$
are isotopic.

Hence, let $i\ge1$. For notational simplicity, let $L:=K_{m,a,i}^{f}$,
and consider the rightmost $\sigma_{i}$ of the braid $(\sigma_{1}\cdots\sigma_{2m})^{a}\sigma_{1}\cdots\sigma_{i}$.
This crossing gives rise to an unoriented skein exact triangle that
involves $L$, its oriented resolution $L_{o}$, and its unoriented
resolution $L_{u}$, and let us orient $L_{o}$ and $L_{u}$ such
that the merge and split bands respect the orientations of the links
$L,L_{o},L_{u}$.
\begin{rem}
We find the following helpful to recall: first, the order of the three
band maps (compare \cite[Lemma~5.2]{MR4843752}) in an unoriented
skein exact triangle is a cyclic permutation of nonorientable, split,
and merge. Second, there exists some permutation $B_{0},B_{1},B_{2}$
of the three links such that $|B_{0}|=|B_{1}|=|B_{2}|-1$. In this
case, the band between $B_{0}$ and $B_{1}$ is nonorientable.
\end{rem}

The punchline is that we can use the induction hypotheses (\ref{enu:induct1})~and~(\ref{enu:induct2})
for $L_{o}$ and $L_{u}$. First, indeed $L_{o}=K_{m,a,i-1}^{f}$.
Next, by \cite[Lemma~A.3~(1)~and~(2)]{ren2024khovanovv3} (also see
\cite{MR2308944,MR2492301} and \cite[Figures~15--20]{MR3116306}),
$L_{u}$ is isotopic to $K_{m',a',i'}^{f}$, $K_{m',a',i'}^{f}\sqcup U$,
$K_{m',a',i'}^{f}{}'$, or $K_{m',a',i'}^{f}{}'\sqcup U$ for some
$m',a',i'$ such that $(f,m',a',i')\in\mathrm{Ind}$ and $m'<m$.
Statements~(\ref{enu:induct1})~and~(\ref{enu:induct2}) hold for
$K_{m',a',i'}^{f}$ and hence for $K_{m',a',i'}^{f}\sqcup U$. Since
$K_{m',a',i'}^{f}{}'=r((rK)_{m',a',i'}^{f})$ where $r$ means orientation
reverse and $rK$ has a negative diagram of writhe $w$, Statements~(\ref{enu:induct1})~and~(\ref{enu:induct2})
hold for $K_{m',a',i'}^{f}{}'$ and $K_{m',a',i'}^{f}{}'\sqcup U$
as well.

We will use the exact triangles for $\overline{Kh}$ and $\overline{Kh}_{Lee}$
to prove the induction step. For this, let us study the homological
degrees of the maps involved. First, $x_{L}$ maps to a nonzero multiple
of $x_{L_{o}}$, and so the map from $L$ to $L_{o}$ has homological
degree $0$. Hence the exact triangles are
\begin{gather}
\cdots\to\overline{Kh}^{h+d}(L_{u})\to\overline{Kh}^{h}(L)\to\overline{Kh}^{h}(L_{o})\to\overline{Kh}^{h+d+1}(L_{u})\to\cdots\label{eq:long-exact}\\
\cdots\to\overline{Kh}_{Lee}^{h+d}(L_{u})\to\overline{Kh}_{Lee}^{h}(L)\to\overline{Kh}_{Lee}^{h}(L_{o})\to\overline{Kh}_{Lee}^{h+d+1}(L_{u})\to\cdots\label{eq:long-exact-lee}
\end{gather}
for some $d\in\mathbb{Z}$. Lemma~\ref{lem:technical} is the key
lemma; we prove it at the end.
\begin{lem}[{\cite[Lemma A.3 (4)]{ren2024khovanovv3}}]
\label{lem:technical}We have $d\ge0$. Furthermore, if $a<2m$,
then $d>0$.
\end{lem}

Let us show Statement~(\ref{enu:induct1}) for $L$. Let $h>0$;
then since $h+d\ge h>0$, $\overline{Kh}^{h+d}(L_{u})$ and $\overline{Kh}^{h}(L_{o})$
vanish. Hence $\overline{Kh}^{h}(L)$ also vanishes since (\ref{eq:long-exact})
is exact.

Let us show Statement~(\ref{enu:induct2}) for $L$. If $a<2m$,
then $d>0$, and so $\overline{Kh}^{d}(L_{u})$ and $\overline{Kh}^{d+1}(L_{u})$
vanish. Hence
\[
0\to\overline{Kh}^{0}(L)\to\overline{Kh}^{0}(L_{o})\to0\ \mathrm{and}\ 0\to\overline{Kh}_{Lee}^{0}(L)\to\overline{Kh}_{Lee}^{0}(L_{o})\to0
\]
are exact since (\ref{eq:long-exact})~and~(\ref{eq:long-exact-lee})
are exact. Thus Statement~(\ref{enu:induct2}) follows.

If $a=2m$, then $|L|=i+1$ and $|L_{o}|=i$, and so $\overline{Kh}_{Lee}^{h}(L_{o})\to\overline{Kh}_{Lee}^{h+d+1}(L_{u})$
is a nonorientable band map which is zero. Furthermore, $\overline{Kh}^{d+1}(L_{u})$
vanishes since $d\ge0$. Hence, 
\begin{gather}
\overline{Kh}^{-1}(L_{o})\to\overline{Kh}^{d}(L_{u})\xrightarrow{G}\overline{Kh}^{0}(L)\to\overline{Kh}^{0}(L_{o})\to0,\label{eq:kh1}\\
0\to\overline{Kh}_{Lee}^{d}(L_{u})\to\overline{Kh}_{Lee}^{0}(L)\to\overline{Kh}_{Lee}^{0}(L_{o})\to0\label{eq:khlee1}
\end{gather}
are exact since (\ref{eq:long-exact})~and~(\ref{eq:long-exact-lee})
are exact. Thus we have 
\begin{align*}
\dim\overline{Kh}_{Lee}^{0}(L) & \le\dim\overline{Kh}^{0}(L)\\
 & =\dim\overline{Kh}^{d}(L_{u})+\dim\overline{Kh}^{0}(L_{o})-\dim({\rm ker}G) & \mathrm{(\ref{eq:kh1})\ is\ exact}\\
 & \le\dim\overline{Kh}^{d}(L_{u})+\dim\overline{Kh}^{0}(L_{o})\\
 & =\dim\overline{Kh}_{Lee}^{d}(L_{u})+\dim\overline{Kh}_{Lee}^{0}(L_{o}) & \mathrm{Statement\ (\ref{enu:induct2})\ for\ }L_{u}\ \mathrm{and\ }L_{o}\\
 & =\dim\overline{Kh}_{Lee}^{0}(L), & \mathrm{(\ref{eq:khlee1})\ is\ exact}
\end{align*}
and so the above inequalities are all equalities.

Since the first inequality is an equality, Statement~(\ref{enu:induct2})
holds for $L$. Since the second inequality is an equality, the map
$G:\overline{Kh}^{d}(L_{u})\to\overline{Kh}^{0}(L)$ is injective.
If $L=K_{m,2m,2m}^{0}$, then $L_{u}$ is isotopic to $K_{m-1,2m-2,2m-2}^{0}\sqcup U$
(\cite[Lemma~A.3~(3)]{ren2024khovanovv3}) and $G$ is precisely the
map $\varphi$ from Lemma~\ref{lem:band-map} (in particular, $d=0$).
Hence, Lemma~\ref{lem:band-map} follows.
\begin{proof}[Proof of Lemma~\ref{lem:technical}]
This follows from \cite[Lemma A.3 (4)]{ren2024khovanovv3}, but we
reprove it for completeness. Note that in this proof, ${\rm gr}_{h}$
always denotes the homological grading in $\overline{Kh}_{Lee}$.

Before we prove the lemma, let us study general $L:=K_{m,a,i}^{f}$.
By the \emph{$j$th strand} of a braid (on $(2m+1)$ strands), we
mean the $j$th input strand: e.g.\ if the braid is $\sigma_{1}\cdots\sigma_{i}$,
then $\sigma_{i}$ swaps the first strand and the $(i+1)$th strand.
Note that the $j$th strand might not close up to a link component
in the braid closure of $D_{m,a,i}$ and hence in $L$. Note that,
on the other hand, it does close up if $(m,a,i)=(m,2m,2m)$.

Let $J\subset\{1,\cdots,2m+1\}$ be a set of strands such that the
$j$th strands for $j\in J$ form a sublink of the braid closure of
$D_{m,a,i}$. Denote this sublink as $D_{m,a,i}^{J}$, and denote
the corresponding sublink of $L$ as $L^{J}$. Let $x_{L}^{J}$ be
the Lee generator of $L$ where $L\setminus L^{J}$ is oriented as
$K$ and $L^{J}$ is oriented oppositely.

Let $C:=K_{m,2m,2m}^{f}$. As an unoriented link, $C$ is an $(f+1)$-framed
$(2m+1)$-cable of $K$. Using the notation of Lemma~\ref{lem:orientation-op},
$x_{C}^{J}$ is $x_{\mathfrak{o}_{|J|}}$, and so by Lemma~\ref{lem:orientation-op}
we have (recall that $f\le0$)
\begin{align}
{\rm gr}_{h}x_{C}^{J}-{\rm gr}_{h}x_{C} & =(f+1)((2m+1)^{2}-(2m+1-2|J|)^{2})/2\nonumber \\
 & \le((2m+1)^{2}-(2m+1-2|J|)^{2})/2.\label{eq:third}
\end{align}

For link diagrams $D_{1},D_{2}$, define $cr(D_{1},D_{2})$ as the
set of crossings between $D_{1}$ and $D_{2}$. Then $cr(L^{J},L\setminus L^{J})\subset cr(C^{J},C\setminus C^{J})$,
and their difference comes from the difference between $cr(D_{m,a,i}^{J},D_{m,a,i}\setminus D_{m,a,i}^{J})\subset cr(D_{m,2m,2m}^{J},D_{m,2m,2m}\setminus D_{m,2m,2m}^{J})$,
which are all positive crossings. Hence, if we let $\#Cr:=|cr(D_{m,2m,2m}^{J},D_{m,2m,2m}\setminus D_{m,2m,2m}^{J})|-|cr(D_{m,a,i}^{J},D_{m,a,i}\setminus D_{m,a,i}^{J})|\ge0$,
then by Lemma~\ref{lem:hgrading}, we have
\begin{equation}
{\rm gr}_{h}x_{L}^{J}-{\rm gr}_{h}x_{L}+\#Cr=2\ell k(L^{J},L\setminus L^{J})+\#Cr=2\ell k(C^{J},C\setminus C^{J})={\rm gr}_{h}x_{C}^{J}-{\rm gr}_{h}x_{C}.\label{eq:cr-iden}
\end{equation}
Recall that $D_{m,2m,2m}$ is the braid closure of $(\sigma_{1}\cdots\sigma_{2m})^{2m+1}$.
If $a<2m$, then at least one of the crossings of the last $(\sigma_{1}\cdots\sigma_{2m})$
of $(\sigma_{1}\cdots\sigma_{2m})^{2m+1}$ contributes to $\#Cr$,
and so $\#Cr\ge1$.

Let us now prove the lemma. Recall that $L=K_{m,a,i}^{f}$, $i\ge1$,
$L_{o}=K_{m,a,i-1}^{f}$, $L_{u}$ is isotopic to $K_{m',a',i'}^{f}$,
$K_{m',a',i'}^{f}\sqcup U$, $K_{m',a',i'}^{f}{}'$, or $K_{m',a',i'}^{f}{}'\sqcup U$,
and $m'<m$. We consider two cases: the map $\overline{Kh}_{Lee}^{\bullet}(L)\to\overline{Kh}_{Lee}^{\bullet}(L_{o})$
is either a merge band map or a split band map. If it is a merge band
map, then $\overline{Kh}_{Lee}^{\bullet+d}(L_{u})\to\overline{Kh}_{Lee}^{\bullet}(L)$
is a split band map. The image of the Lee generator $x_{L_{u}}$ is
a nonzero multiple of $x_{L}^{J}$ for some subset $J\subset\{1,\cdots,2m+1\}$
such that $|J|=m-m'$. Then, we have (recall ${\rm gr}_{h}x_{L}=-(2m+1)^{2}/2+1/2$
by definition)
\begin{align*}
{\rm gr}_{h}x_{L}^{J}+\#Cr & ={\rm gr}_{h}x_{C}^{J}-{\rm gr}_{h}x_{C}-(2m+1)^{2}/2+1/2 & \mathrm{(\ref{eq:cr-iden})}\\
 & \le((2m+1)^{2}-(2m+1-2|J|)^{2})/2-(2m+1)^{2}/2+1/2 & \mathrm{(\ref{eq:third})}\\
 & =-(2m'+1)^{2}/2+1/2={\rm gr}_{h}x_{L_{u}}={\rm gr}_{h}x_{L}^{J}+d
\end{align*}
and so $d\ge0$ since $\#Cr\ge0$. If $a<2m$, then $d\ge1$ since
$\#Cr\ge1$.

If $\overline{Kh}_{Lee}^{\bullet}(L)\to\overline{Kh}_{Lee}^{\bullet}(L_{o})$
is a split band map, then $\overline{Kh}_{Lee}^{\bullet}(L_{o})\to\overline{Kh}_{Lee}^{\bullet+d+1}(L_{u})$
is a merge band map. Again, there exists some $J\subset\{1,\cdots,2m+1\}$
such that $|J|=m-m'$ and $x_{L_{o}}^{J}$ maps to a nonzero multiple
of $x_{L_{u}}$. Then, Equation~(\ref{eq:cr-iden}) for $(L_{o},m,a,i-1)$
instead of $(L,m,a,i)$ says that if we let $\#Cr:=|cr(D_{m,2m,2m}^{J},D_{m,2m,2m}\setminus D_{m,2m,2m}^{J})|-|cr(D_{m,a,i-1}^{J},D_{m,a,i-1}\setminus D_{m,a,i-1}^{J})|$
then we have
\begin{equation}
{\rm gr}_{h}x_{L_{o}}^{J}-{\rm gr}_{h}x_{L_{o}}+\#Cr={\rm gr}_{h}x_{C}^{J}-{\rm gr}_{h}x_{C}.\label{eq:split}
\end{equation}
Consider the rightmost $\sigma_{i}$ of the braid $(\sigma_{1}\cdots\sigma_{2m})^{a}\sigma_{1}\cdots\sigma_{i}$.
Exactly one of the two strands that this crossing swaps is in $J$,
and this crossing is not in $L_{o}$. Hence, it contributes to $\#Cr$,
and so $\#Cr\ge1$. If $a<2m$,\footnote{Note that in fact if $\overline{Kh}_{Lee}^{\bullet}(L)\to\overline{Kh}_{Lee}^{\bullet}(L_{o})$
is a split band map, then we must have $a<2m$.} then at least one of the crossings of the last $(\sigma_{1}\cdots\sigma_{2m})$
of $(\sigma_{1}\cdots\sigma_{2m})^{2m+1}$ contributes to $\#Cr$,
and this crossing is different from the above crossing. Hence $\#Cr\ge2$.

Thus, similarly to the case where $\overline{Kh}_{Lee}^{\bullet}(L)\to\overline{Kh}_{Lee}^{\bullet}(L_{o})$
is a merge band map, Equation~(\ref{eq:split}) and Inequality~(\ref{eq:third})
give ${\rm gr}_{h}x_{L_{o}}^{J}+\#Cr\le{\rm gr}_{h}x_{L_{u}}={\rm gr}_{h}x_{L_{o}}^{J}+d+1$,
and so $d\ge0$ since $\#Cr\ge1$. If $a<2m$, then $\#Cr\ge2$, and
so $d\ge1$.
\end{proof}

\subsection*{Acknowledgements}

We thank Peter Ozsv\'{a}th for his continuous support and helpful
discussions. Page~\pageref{pagelabel} onwards can be thought of
as expanded notes for a talk for the \href{https://sites.google.com/view/fa24lasagna/home}{Lasagna Reading Seminar},
Princeton, Fall 2024. We thank Robert Lipshitz, Ciprian Manolescu,
Qiuyu Ren, and the organizers and participants of the Lasagna Reading
Seminar for helpful conversations, and Ayodeji Lindblad, Robert Lipshitz,
and Qiuyu Ren for helpful comments on an earlier draft.

\enlargethispage{2\baselineskip}

\bibliographystyle{amsalpha}
\bibliography{/Users/gheehyun/Documents/writings/bib}

\end{document}